\documentclass[reqno]{amsart}
\usepackage{amscd,verbatim,amssymb}
\usepackage[mathscr]{eucal}
\usepackage{color}
\usepackage{verbatim}
\newtheorem{theorem}{Theorem}[section]
\newtheorem{proposition}[theorem]{Proposition}
\newtheorem{lemma}[theorem]{Lemma}
\newtheorem{corollary}[theorem]{Corollary}

\theoremstyle{definition}
\newtheorem{remark}[theorem]{Remark}
\newtheorem{example}[theorem]{Example}

\newtheorem{dfn}[theorem]{Definition}

\newcommand{\ff}{F^{\times}}
\newcommand{\fs}{F^{\times 2}}

\newcommand{\cchar}{\mathrm{char}}

\newcommand{\id}{\mathrm{id}}

\newcommand{\invo}{\overline{\rule{2.5mm}{0mm}\rule{0mm}{4pt}}}

\newcommand{\Pic}{\operatorname{Pic}}

\newcommand{\Br}{\operatorname{Br}}

\newcommand{\Gal}{\operatorname{Gal}}

\newcommand{\gGO}{\operatorname{\mathbf{GO}}}

\newcommand{\gPGU}{\operatorname{\mathbf{PGU}}}

\newcommand{\gSpin}{\operatorname{\mathbf{Spin}}}

\newcommand{\gGm}{\operatorname{\mathbf{G}_m}}
\newcommand{\End}{\operatorname{End}}

\newcommand{\Nrd}{\operatorname{Nrd}}

\newcommand{\ad}{\operatorname{ad}}

\newcommand{\disc}{\operatorname{disc}}
\newcommand{\Hyp}{\operatorname{Hyp}}

\DeclareMathOperator{\gPGO}{\mathbf{PGO}}
\DeclareMathOperator{\GO}{GO}
\DeclareMathOperator{\PGO}{PGO}
\DeclareMathOperator{\gPGSp}{\mathbf{PGSp}}
\DeclareMathOperator{\GU}{GU}
\DeclareMathOperator{\PGU}{PGU}
\DeclareMathOperator{\gAut}{\mathbf{Aut}}

\newcommand{\Z}{\mathbb{Z}}

\newcommand{\wF}{{\widehat F}}
\newcommand{\wK}{\widehat K}
\newcommand{\wB}{\widehat B}
\newcommand{\wtau}{\widehat\tau}

\newcommand{\cO}{\mathcal O}

\title[Outer automorphisms and non-rational groups]{Outer
  automorphisms of adjoint groups of type $\mathsf D$ and 
  non-rational adjoint groups of outer type $\mathsf{A}$}
\keywords{Adjoint group, outer automorphism, involution, hermitian form, similitude, $R$-equivalence, stable rationality}
\subjclass[2010]{Primary 11E57; Secondary 20G15}
\date{June 3, 2018}
\author[D. Barry]{Demba Barry}
\author[J.-P. Tignol]{Jean-Pierre Tignol}
\address{Facult\'e des Sciences et Techniques de Bamako, BP: E3206 Bamako, Mali  and Departement Wiskunde--Informatica, Universiteit Antwerpen, Belgium}
\email{Barry.Demba@gmail.com}
\address{Universit\'e catholique de Louvain, ICTEAM Institute, Avenue G. Lema\^itre 4, Box L4.05.01, B-1348 Louvain-la-Neuve, Belgium.}
\email{Jean-Pierre.Tignol@uclouvain.be}
\thanks{The first author would like to thank the second author and UCL for their hospitality while the work for this paper was done. He gratefully acknowledges support from the FWO Odysseus Programme ({\it project Explicit Methods in Quadratic Form Theory}). The second author acknowledges support from the Fonds de la Recherche Scientifique--FNRS under grant n$^\circ$~J.0149.17.}
\begin{document}
\maketitle
\begin{abstract}  
 For a classical group $G$ of type $\mathsf D_n$ over a field $k$ of
 characteristic different from $2$, we show the existence of a
 finitely generated regular extension $R$ of $k$ such that $G$ admits
 outer automorphisms over $R$. Using this result and a construction of
 groups of type $\mathsf A$ from groups of type $\mathsf D$, we
 construct new examples of groups of type $^2\mathsf A_n$ with
 $n\equiv 3\bmod 4$ and the first examples of type $^2\mathsf A_n$
 with $n\equiv 1\bmod 4$  $(n\geq 5)$ that are  not $R$-trivial, hence
 not rational (nor stably rational). 
\end{abstract}
\section{Introduction}\label{section1}

Two questions concerning algebraic groups of classical type are
addressed in this paper: the existence of outer automorphisms of
adjoint groups of type $\mathsf D$ and the rationality of adjoint
groups of outer type $\mathsf A$. The two questions are related by a
construction of groups of type $\mathsf A$ from groups of type
$\mathsf D$ that we call unitary extension. 

To describe our contribution to the first topic, recall that
when a classical group $G$ of adjoint type $\mathsf{D}_n$ over a field
$k$ of characteristic different from~$2$ is represented as the
connected component of the identity $\gPGO^+(A,\sigma)$ in the group
of automorphisms of a central simple algebra with orthogonal
involution $(A,\sigma)$ of degree~$2n$, then outer automorphisms of
$G$ are induced by improper similitudes of $(A,\sigma)$, i.e.,
elements $g\in A$ such that
\[
\sigma(g)g\in k^\times \quad\text{and}\quad
\Nrd_A(g)=-(\sigma(g)g)^n,
\]
see \cite[Prop.~2.5]{QMT1}. The existence of an improper similitude is
a serious 
constraint on $A$ and $\sigma$: the algebra must be split by the
quadratic extension given by the discriminant of $\sigma$ (see
\cite[(13.38)]{KMRT98}), hence its index is at most~$2$; and if the
discriminant is trivial then $A$ must be split.
Nevertheless, we show:

\begin{proposition}
  \label{prop:improper}
  Let $\sigma$ be an orthogonal involution on a central simple algebra
  $A$ of degree~$2n$ over a field $k$ of characteristic different
  from~$2$. If $A$ is not split and the discriminant of $\sigma$ is
  not trivial, there exists a finitely generated regular extension $R$
  of $k$ such that 
  the algebra with involution $(A_R,\sigma_R)$ obtained from
  $(A,\sigma)$ by scalar extension to $R$ admits improper similitudes
  and $A_R$ is not split.
\end{proposition}

The proof\footnote{The authors are grateful to the referee for
  suggesting this line of proof, which significantly simplifies their
  initial approach, and to Merkurjev for advice on the proof of
  Proposition~\ref{prop:kerBr}.} shows that one can take for $R$ the
function field of the 
connected component of improper similitudes in the group of
automorphisms of $(A,\sigma)$: see Section~\ref{sec:improper}.
\medbreak

In the second part of the paper, we investigate the rationality
problem for the underlying variety of adjoint linear algebraic groups
of type $^2\mathsf{A}_n$. Voskresenski\u\i\ and Klyachko \cite[Cor.~of
Th.~8]{VK85} have shown that this variety is rational if $n$ is
even. By contrast, examples of adjoint groups of type $^2\mathsf{A}_n$
for $n\equiv3\bmod4$ that are not rational have been given by
Merkurjev \cite{Mer96} and by 
Berhuy--Monsurr\`o--Tignol \cite{BMT04}, using Manin's $R$-equivalence
and Merkurjev's computation of the group of $R$-equivalence classes
of adjoint classical groups \cite{Mer96}. On one hand we will use
Proposition~\ref{prop:improper} to expand the range of these examples,
and 
on another hand we will provide the first examples of adjoint groups
of type $^2\mathsf{A}_n$ with arbitrary $n\equiv1\bmod 4$ ($n\geq5$)
that are not $R$-trivial, hence not rational (nor stably
rational). These examples 
are based on the adjoint groups of type $\mathsf D$ that have outer
automorphisms but no outer automorphisms of order~$2$ found by
Qu\'eguiner-Mathieu and Tignol \cite{QMT1}.

To explain our construction in more detail, recall that
adjoint groups of outer type $\mathsf A$ over a field $F$ can be
represented as groups of automorphisms $\gPGU(B,\tau)$ of central
simple algebras with unitary involution $(B,\tau)$ over separable
quadratic field extensions $K/F$. We consider in particular the case
where $(B,\tau)$ is obtained from a central simple $F$-algebra with
orthogonal or symplectic involution $(A,\sigma)$ as
\[
(B,\tau)=(A,\sigma)\otimes_F(K,\iota),
\]
where $\iota$ is the nontrivial $F$-automorphism of $K$. We then say
$(B,\tau)$ is a \emph{unitary extension} of $(A,\sigma)$. In
Section~\ref{sec:unitext} we give a necessary and sufficient condition
for $(B,\tau)$ to be hyperbolic (excluding one exceptional case); see Theorem~\ref{thm:hypcond}.

Of special interest are \emph{generic} unitary extensions, where
$F=k(x)$ is a rational function field in one variable over a field $k$
of characteristic different from $2$, $K=F(\sqrt x)$, and $A$ is
defined over the field $k$ of constants. (Generic unitary extensions
are also used in \cite[Sec.~4.3]{QMT1}.) In
Section~\ref{sec:examples} we show:

\begin{theorem}
  \label{thm:A3}
  Let $(A,\sigma)$ be a central simple algebra with orthogonal
  involution of degree multiple of~$4$ over a field $k$ of
  characteristic~$0$. If $A$ is not split and the 
  discriminant of $\sigma$ is not trivial, then for the generic
  unitary extension $(B,\tau)$ of $(A,\sigma)$ the group
  $\gPGU(B,\tau)$ is not $R$-trivial, i.e., there exists a field
  extension $E$ of $F$ such that the group of $E$-rational points
  $\gPGU(B,\tau)(E)$ has more than one $R$-equivalence class.
\end{theorem}

It follows that the group $\gPGU(B,\tau)$ is not rational, nor even
stably rational, see~\cite[Sec.~4]{CTS77}. It is a group of adjoint
type $^2\mathsf{A}_n$ with $n\equiv3\bmod4$. The proof of
Theorem~\ref{thm:A3} is given in Subsection~\ref{subsec:3mod4}.
\smallbreak

Section~3 of \cite{QMT1} yields examples of central simple algebras
with orthogonal involution $(A,\sigma)$ of degree 
$\deg A\equiv2\bmod4$ 
that have improper similitudes, none of them being square-central. We
show in Example~\ref{ex:main} that for their generic unitary extension
$(B,\tau)$ the group $\gPGU(B,\tau)$ is not $R$-trivial. We
thus obtain examples of adjoint groups of type $^2\mathsf{A}_n$ that
are not rational nor stably rational for every integer $n\geq5$ with
$n\equiv1\bmod4$.
\smallbreak

The proof of Theorem~\ref{thm:A3} is prepared in Section~\ref{sec:Req}
by Theorem~\ref{thm:PGU/R}, which yields a computation in terms of
$(A,\sigma)$ of the group of $R$-equivalence classes in the group of
rational points of $\gPGU(B,\tau)$ over the completion of $F$ for the
$x$-adic valuation.

\subsection*{Notation}
We generally follow the notation and terminology of \cite{KMRT98}. The
characteristic of the base field is always assumed to be different
from~$2$. If $A$ is a central simple algebra of even degree over a
field $k$ and $\sigma$ is an orthogonal involution on $A$, we write
$\GO(A,\sigma)$ for the group of similitudes of $(A,\sigma)$,
\[
\GO(A,\sigma)=\{g\in A\mid \sigma(g)g\in k^\times\}.
\]
For $g\in \GO(A,\sigma)$ we let $\mu(g)=\sigma(g)g$ be the
\emph{multiplier} of $g$. The group of proper similitudes is
\[
\GO^+(A,\sigma)=\{g\in\GO(A,\sigma)\mid \Nrd_A(g)=\mu(g)^{(\deg
  A)/2}\}
\]
and we write 
\[
\GO^-(A,\sigma)=\GO(A,\sigma)\setminus \GO^+(A,\sigma) =
\{g\in\GO(A,\sigma)\mid \Nrd_A(g)=-\mu(g)^{(\deg A)/2}\}
\]
for the coset of improper similitudes (which may be empty). The
corresponding sets of multipliers are denoted as follows:
\[
G(A,\sigma)=\{\mu(g)\mid g\in\GO(A,\sigma)\},\quad
G^\pm(A,\sigma)=\{\mu(g)\mid g\in\GO^\pm(A,\sigma)\}.
\]
If $\delta\in k^\times$ represents the discriminant of $\sigma$, the
quaternion algebra $\bigl(\delta,\mu(g)\bigr)_k$ is Brauer-equivalent
to $k$ if $g$ is proper, and to $A$ if $g$ is improper;
see~\cite[Th.~A]{MT95} or \cite[(13.38)]{KMRT98}. Therefore,
\begin{equation}
  \label{eq:propsimnothing}
  G^+(A,\sigma)\cap G^-(A,\sigma) = \varnothing \qquad\text{if $A$ is
    not split.}
\end{equation}
The group of projective proper similitudes of $(A,\sigma)$ is
$\PGO^+(A,\sigma)=\GO^+(A,\sigma)/k^\times$. It is the group of
$k$-rational points of the algebraic group $\gPGO^+(A,\sigma)$, which
is the connected component of the identity in the group
$\gPGO(A,\sigma)=\gAut(A,\sigma)$ of automorphisms of $(A,\sigma)$. It
is a group of adjoint type $\mathsf{D}_n$ if $\deg A=2n$ with
$n\geq2$. 

For every field $\ell$ containing $k$ we write $A_\ell$ for the
$\ell$-algebra $A\otimes_k\ell$ and $\sigma_\ell$ for the involution
$\sigma\otimes\id_\ell$ on $A_\ell$, and we let
$(A,\sigma)_\ell=(A_\ell,\sigma_\ell)$. If $\ell$ is a finite-degree
extension of $k$, we let $N(\ell/k)=N_{\ell/k}(\ell^\times)$ be the
group of norms. Let $\Hyp(A,\sigma)\subset k^\times$ be the subgroup
generated by the norm groups $N(\ell/k)$ where $\ell$ runs over the
finite-degree field extensions of $k$ such that $(A,\sigma)_\ell$ is
hyperbolic. The following canonical isomorphism due to Merkurjev
\cite[Th.~1]{Mer96} yields a description 
of the group of $R$-equivalence classes of $\PGO^+(A,\sigma)$:
\begin{equation}
\label{eq:MerkPG0/R}
\PGO^+(A,\sigma)/R \simeq
G^+(A,\sigma)\big/\bigl(k^{\times2}\cdot\Hyp(A,\sigma)\bigr). 
\end{equation}
In particular, if $(A,\sigma)$ is hyperbolic, then
$\Hyp(A,\sigma)\supset N(k/k) = k^\times$, hence
\begin{equation}
  \label{eq:orthyp}
  G(A,\sigma)=G^+(A,\sigma)= \Hyp(A,\sigma)=k^\times
   \quad\text{and}\quad \PGO^+(A,\sigma)/R = 1
   \quad\text{if
    $(A,\sigma)$ is hyperbolic.}
\end{equation}

Corresponding notions are defined for unitary involutions: if $B$ is a
central simple algebra over a field $K$ and $\tau$ is a unitary
involution on $B$, i.e., an involution that does not leave $K$
elementwise fixed, let $F\subset K$ be the subfield of fixed elements
and
\[
\GU(B,\tau)=\{g\in B\mid \tau(g)g\in F^\times\},\qquad
G(B,\tau)=\{\tau(g)g\mid g\in \GU(B,\tau)\}\subset F^\times.
\]
The group of projective similitudes
$\PGU(B,\tau)=\GU(B,\tau)/K^\times$ is the group of $F$-rational
points of the algebraic group $\gPGU(B,\tau)=\gAut_K(B,\tau)$, which
is a group of adjoint type $^2\mathsf{A}_{n-1}$ over $F$ if $\deg
B=n>2$. The group $\Hyp(B,\tau)\subset F^\times$ is defined as in the
orthogonal case, 
and Merkurjev's canonical isomorphism takes the form
\begin{equation}
\label{eq:MerkPGU/R}
\PGU(B,\tau)/R \simeq
G(B,\tau)\big/\bigl(N(K/F)\cdot\Hyp(B,\tau)\bigr). 
\end{equation}
As in the orthogonal case, we have
\begin{equation}
  \label{eq:unithyp}
  G(B,\tau)=\Hyp(B,\tau)=F^\times
 \quad\text{and}\quad
    \PGU(B,\tau)/R=1
  \quad\text{if $(B,\tau)$ is hyperbolic.}
\end{equation}

\section{Improper similitudes}
\label{sec:improper}

Throughout this section, $A$ is a central simple algebra of
degree~$2n$ over an arbitrary field $k$ of characteristic different
from~$2$ and $\sigma$ is an orthogonal involution on $A$. If $A$ is a
quaternion algebra, then $(A,\sigma)$ admits improper similitudes
(see~\cite[(12.25)]{KMRT98}), hence Proposition~\ref{prop:improper}
holds with $R=k$. We may therefore assume throughout $n\geq2$, so
$\gPGO^+(A,\sigma)$ is a semisimple linear algebraic group.

Recall from~\cite[\S23.B]{KMRT98} that $\gPGO(A,\sigma)$ has two
connected components. Write $X=\gPGO^-(A,\sigma)$ for the non-identity
component. It is a $\gPGO^+(A,\sigma)$-torsor whose rational points
consist of inner automorphisms induced by improper similitudes of
$(A,\sigma)$. Therefore, $X$ is an affine, smooth, geometrically
connected $k$-variety, and its function field $k(X)$ is a finitely
generated regular extension of $k$. By definition, $X$ has rational
points over $k(X)$, hence $(A_{k(X)},\sigma_{k(X)})$ admits improper
similitudes. To establish Proposition~\ref{prop:improper}, it is
therefore sufficient to show:

\begin{proposition}
  \label{prop:kerBr}
  If the discriminant of $\sigma$ is not trivial, then the Brauer
  group map $\Br(k)\to\Br\bigl(k(X)\bigr)$ induced by scalar extension
  is injective.
\end{proposition}

\begin{proof}
  Since $X$ is smooth the map $\Br(X)\to\Br\bigl(k(X)\bigr)$ is
  injective, hence it suffices to show that the map $\Br(k)\to\Br(X)$
  is injective. 

  Let $k_s$ be a separable closure of $k$ and $\Gamma=\Gal(k_s/k)$ the
  absolute Galois group of $k$. To simplify notations, write $G$ for
  $\gPGO^+(A,\sigma)$, and let $G_s$ (resp.\ $X_s$) denote the
  algebraic group over $k_s$ (resp.\ algebraic variety over $k_s$)
  obtained from $G$ (resp.\ $X$) by base change from $k$ to
  $k_s$. Since $X_s(k_s)\neq\varnothing$, the variety $X_s$ is
  isomorphic to the underlying variety of $G_s$. It then follows from
  a theorem of Rosenlicht~\cite[Th.~3]{Ros61} that every invertible
  regular function 
  on $X_s$ is constant. Therefore, the Hochschild--Serre spectral
  sequence $H^p\bigl(\Gamma,H^q(X_s,\gGm)\bigr)\Rightarrow
  H^n(X,\gGm)$ yields the following exact sequence of low-degree terms
  (see \cite[Lemme~6.3(i)]{San81}):
  \[
  0\to\Pic(X)\to(\Pic X_s)^\Gamma \to \Br(k) \to
  \Br(X).
  \]
  Now, by \cite[Lemme~6.7]{San81} we have $(\Pic X_s)^\Gamma\simeq
  (\Pic G_s)^\Gamma$ and, by \cite[Lemme~6.9]{San81} (see
  also~\cite[(31.21)]{KMRT98}), $\Pic G_s$ can be identified with the
  dual $Z_s^*$ of the center $Z_s$ of the simply connected cover
  $\gSpin(A_s,\sigma_s)$ of $G_s$. Since the discriminant of $\sigma$
  is not trivial, $\Gamma$ acts non-trivially on $Z_s^*$, and we have
  $(\Pic G_s)^\Gamma\simeq\Z/2\Z$. Therefore, to complete the proof it
  suffices to show that $\Pic(X)\neq0$.

  For this, consider the canonical map $\gGO^-(A,\sigma)\to X$: it
  defines a torsor for $\gGm$ over $X$, hence an element of
  $H^1(X,\gGm)=\Pic(X)$. This element is not trivial because after
  scalar extension to $k_s$ the torsor is isomorphic to
  $\gGO^+(A_s,\sigma_s)\to G_s$. The proof is thus complete.
\end{proof}

\section{Unitary extensions of involutions of the first kind}
\label{sec:unitext}

In this section, $A$ is a central simple algebra over an arbitrary
field $F$ of characteristic different from~$2$ and $\sigma$ is an
$F$-linear involution on $A$ (i.e., an involution that may be
orthogonal or symplectic). Let $K$ be a
quadratic field extension of $F$ and let $\iota$ denote its nontrivial
automorphism. We consider the algebra with unitary involution
\[
(B,\tau)=(A,\sigma)\otimes_F(K,\iota).
\]
In preparation for the next section, where a special case of this
construction will be analyzed, we determine a necessary and sufficient
condition for $(B,\tau)$ to be hyperbolic.

\begin{theorem}
  \label{thm:hypcond}
  If there is an embedding of $F$-algebras with involution $(K,\id)
  \hookrightarrow (A,\sigma)$, then $(B,\tau)$ is hyperbolic. The
  converse holds, except in the case where $A$ is split of
  degree~$2\bmod4$ and $\sigma$ is symplectic.
\end{theorem}

The proof uses the Witt decomposition of involutions. Recall that $A$
can be represented as $\End_DV$ for some vector space $V$ over a
division algebra $D$; then $\sigma$ is adjoint to a nondegenerate
hermitian (or skew-hermitian) form $h$ on $V$ with respect to some
involution of the first kind on $D$. The space $(V,h)$ has a
decomposition
\[
(V,h)\simeq(V_0,h_0)\perp(V_1,h_1)
\]
with $h_0$ anisotropic and $h_1$ hyperbolic, which is reflected in a
so-called orthogonal sum decomposition of $(\End_DV,\ad_h)$ into
$(\End_DV_0,\ad_{h_0}) \boxplus (\End_DV_1,\ad_{h_1})$, see
\cite[Sec.~1.4]{BFT07}. Thus, we may find a decomposition
\[
(A,\sigma)\simeq(A_0,\sigma_0)\boxplus(A_1,\sigma_1)
\]
where $A_0$, $A_1$ are central simple $F$-algebras Brauer-equivalent
to $A$ (if they are not $\{0\}$), where $\sigma_0$, $\sigma_1$ are
involutions of the same type 
as $\sigma$, and where $\sigma_0$ is anisotropic (which means that
$\sigma_0(a)a=0$ implies $a=0$) and $\sigma_1$ is hyperbolic (which
means there is an idempotent $e\in A_1$ such that $\sigma_1(e)=1-e$).

\begin{proof}[Proof of Theorem~\ref{thm:hypcond}]
  Let $K=F(u)$ where $u^2=a\in F^\times$. If $(K,\id)$ embeds into
  $(A,\sigma)$, we may find $s\in A$ such that $\sigma(s)=s$ and
  $s^2=a$. Consider then
  \[
  e={\textstyle\frac12}(1\otimes1+s\otimes u^{-1})\in B.
  \]
  Computation shows that $e^2=e$ and $\tau(e)=1-e$, hence $(B,\tau)$
  is hyperbolic.

  For the converse, suppose we are not in the situation where $A$ is
  split of degree~$2\bmod4$ with $\sigma$ symplectic, and consider a
  Witt decomposition
  $(A,\sigma)\simeq(A_0,\sigma_0)\boxplus(A_1,\sigma_1)$ with
  $\sigma_0$ anisotropic and $\sigma_1$ hyperbolic. If $\sigma$ is
  symplectic and $A$ is split, then $A=A_1$ because symplectic
  involutions on split algebras are hyperbolic, and $\deg A$ is
  assumed to 
  be divisible by~$4$. Likewise, if $\sigma$ is symplectic and $A$ is
  not split, then $\deg A_1$ is divisible by~$4$ because $\sigma_1$ is
  adjoint to a hyperbolic form over a noncommutative division
  algebra. By \cite[Th.~2.2]{BST93} it follows that in all cases
  (including the case where $\sigma$ is orthogonal) there is a central
  simple $F$-algebra with involution $(A'_1,\sigma'_1)$ such that
  \[
  (A_1,\sigma_1) \simeq (M_2(F),\theta) \otimes_F (A'_1,\sigma'_1)
  \]
  with $\theta$ the hyperbolic orthogonal involution defined by
  \[
  \theta
  \begin{pmatrix}
    x_{11}& x_{12}\\ x_{21}& x_{22}
  \end{pmatrix}=
  \begin{pmatrix}
    x_{22}&x_{12}\\ x_{21}& x_{11}
  \end{pmatrix}
  \qquad\text{for $x_{11}$, $x_{12}$, $x_{21}$, $x_{22}\in F$.}
  \]
  Then $A_1$ contains
  $s_1=\bigl(\begin{smallmatrix}0&1\\
    a&0\end{smallmatrix}\bigr)\otimes1$, which satisfies 
  $\sigma_1(s_1)=s_1$ and $s_1^2=a$. We next show that if $(B,\tau)$
  is hyperbolic then $A_0$ contains an element $s_0$ such that
  $\sigma_0(s_0)=s_0$ and $s_0^2=a$; then
  $s=\bigl(\begin{smallmatrix}s_0 & 
    0 \\ 0& s_1\end{smallmatrix}\bigr)\in A_0\boxplus A_1=A$ satisfies $\sigma(s)=s$
  and $s^2=a$, hence mapping $u\mapsto s$ defines an embedding
  $(K,\id)\hookrightarrow(A,\sigma)$.

  Note that $(B,\tau)=[(A_0,\sigma_0)\otimes(K,\iota)] \boxplus
  [(A_1,\sigma_1)\otimes(K,\iota)]$, and the second term on the right
  side is hyperbolic because $\sigma_1$ is hyperbolic. Therefore, the
  hypothesis that $(B,\tau)$ is hyperbolic implies
  $(A_0,\sigma_0)\otimes(K,\iota)$ is hyperbolic. We may then find
  $e=e_1\otimes1+e_2\otimes u\in A_0\otimes_FK$ such that $e^2=e$ and
  $(\sigma_0\otimes\iota)(e)=1-e$, or equivalently
  \begin{equation}
  \label{eq:conde}
  (\sigma_0\otimes\iota)(e)=1-e\qquad\text{and}\qquad
  (\sigma_0\otimes\iota)(e)e=0.
  \end{equation}
  These conditions yield
  \begin{equation}
    \label{eq:conde1}
    \sigma_0(e_1)=1-e_1
  \end{equation}
  and
  \begin{equation}
    \label{eq:conde2}
    \sigma_0(e_1)e_1=a\sigma_0(e_2)e_2.
  \end{equation}
  Now, consider the right ideal $I=\{x\in A_0\mid e_2x=0\}$. By
  \cite[Cor.~1.8]{BST93} we may find $f\in A_0$ such that
  $\sigma_0(f)=f=f^2$ and $I=fA_0$, because $\sigma_0$ is
  anisotropic. Since $e_2f=0$, multiplying~\eqref{eq:conde2} on the
  left and on the right by $f$ yields $\sigma_0(e_1f)e_1f=0$, hence
  $e_1f=0$ because $\sigma_0$ is anisotropic. By~\eqref{eq:conde1} we
  have
  \[
  fe_1\sigma_0(fe_1)=f(e_1-e_1^2)f.
  \]
  The right side is~$0$ since $e_1f=0$, hence $\sigma_0(fe_1)=0$
  because $\sigma_0$ is anisotropic. By~\eqref{eq:conde1} again, it
  follows that $(1-e_1)f=0$, hence $f=0$ since $e_1f=0$. Therefore,
  $I=\{0\}$, hence $e_2$ is invertible and we may set
  $s_0=e_1e_2^{-1}\in A_0$. From~\eqref{eq:conde2} it follows that
  $s_0^2=a$. Now, \eqref{eq:conde} also yields
  $\sigma_0(e_1)e_2=\sigma_0(e_2)e_1$, hence $\sigma_0(s_0)=s_0$.

  We have thus proved the existence of an embedding
  $(K,\id)\hookrightarrow(A,\sigma)$ when $(B,\tau)$ is hyperbolic,
  setting aside the case where $A$ is split of degree~$2\bmod4$ and
  $\sigma$ is symplectic. Note that in the exceptional case $(B,\tau)$
  is hyperbolic since $(A,\sigma)$ is hyperbolic; but every symmetric
  element in $A$ is a root of an odd-degree ``pfaffian'' polynomial
  (see \cite[(2.9)]{KMRT98}), hence there is no $s\in A$ such that
  $\sigma(s)=s$ and $s^2\in \ff\setminus\fs$.
\end{proof}

\begin{remark}
  The proof above is a slight modification of \cite[Th.~3.3]{BST93},
  where the existence of an embedding $(K,\iota)\hookrightarrow
  (A,\sigma)$ is shown to be equivalent to the hyperbolicity of
  $(A,\sigma)\otimes_F(K,\id)$, except when $A$ is split, $\sigma$ is
  orthogonal and its Witt index is odd.
\end{remark}

\section{$R$-equivalence on projective unitary groups}
\label{sec:Req}

In this section, we consider a special case of unitary extension.
Throughout the section, $(A,\sigma)$ is a central simple algebra with
orthogonal or symplectic involution over a field $k$ of characteristic
zero. We let $\widehat F=k((x))$ be the field of formal Laurent
series in one indeterminate over $k$, and $\widehat K=\widehat F(\xi)$
where $\xi^2=x$, hence $\widehat K=k((\xi))$. We write $\iota$ for the
nontrivial automorphism 
of $\widehat K/\widehat F$ and consider the algebra with involution
\begin{equation}
\label{eq:Btau}
(\widehat B,\widehat\tau)=(A,\sigma)\otimes_k(\widehat K,\iota).
\end{equation}
Thus, $\widehat B=A((\xi))$ (with $\xi$ centralizing $A$) and
\[
\widehat\tau\bigl(\sum_{i=r}^\infty a_i\xi^i\bigr) = \sum_{i=r}^\infty
\sigma(a_i)(-\xi)^i \qquad\text{for $a_i\in A$, $i=r$, $r+1$, \ldots}
\]
The $\widehat K$-algebra $\widehat B$ is central simple, and
$\widehat\tau$ is a unitary involution on $\widehat B$. 
Our goal is to compute $\PGU(\widehat B,\widehat\tau)/R$ in terms of
$(A,\sigma)$, using Merkurjev's canonical
isomorphism~\eqref{eq:MerkPGU/R}.
\medbreak

As a first step, we show that the trivial hyperbolic cases
(see~\eqref{eq:orthyp} and \eqref{eq:unithyp}) are related:

\begin{proposition}
  \label{prop:isohyp}
  The statements \emph{(a)} and \emph{(b)} (resp.\ \emph{(a')} and
  \emph{(b')}) are equivalent:
  \begin{align*}
    &\text{\emph{(a)}\quad$(A,\sigma)$ is isotropic,}&
    &\text{\emph{(a')}\quad$(A,\sigma)$ is hyperbolic,}\\
    &\text{\emph{(b)}\quad$(\wB,\wtau)$ is isotropic,}&
    &\text{\emph{(b')}\quad$(\wB,\wtau)$ is hyperbolic.}
  \end{align*}
  Similarly, if $C$ is a central simple algebra over a quadratic field
  extension $\ell$ of $k$ and $\rho$ is a unitary involution on $C$
  fixing $k$, the statements \emph{(c)} and \emph{(d)} (resp.\
  \emph{(c')} and \emph{(d')}) are equivalent:
  \begin{align*}
    &\text{\emph{(c)}\quad$(C,\rho)$ is isotropic,}&
    &\text{\emph{(c')}\quad$(C,\rho)$ is hyperbolic,}\\
    &\text{\emph{(d)}\quad$(C,\rho)_\wF$ is isotropic,}&
    &\text{\emph{(d')}\quad$(C,\rho)_\wF$ is hyperbolic.}
  \end{align*}
\end{proposition}

\begin{proof}
  Since $(A,\sigma)\subset(\wB,\wtau)$, it is clear that
  (a)~$\Rightarrow$~(b) and (a')~$\Rightarrow$~(b'). To see
  (b)~$\Rightarrow$~(a), suppose $y\in\wB$ is nonzero and
  $\wtau(y)y=0$. Write $y$ as a series $y=\sum_{i=r}^\infty a_i\xi^i$
  with coefficients in $A$, with $a_r\neq0$. The coefficient of
  $x^{r}$ in $\wtau(y)y$ is $(-1)^r\sigma(a_r)a_r$, hence
  $\sigma(a_r)a_r=0$. It follows that $\sigma$ is isotropic, proving
  (b)~$\Rightarrow$~(a). 

  To establish (b')~$\Rightarrow$~(a'), consider a Witt decomposition
  $(A,\sigma)\simeq(A_0,\sigma_0)\boxplus(A_1,\sigma_1)$ with
  $\sigma_0$ anisotropic and $\sigma_1$ hyperbolic, as in the proof of
  Theorem~\ref{thm:hypcond}. Then $(A_1,\sigma_1)\otimes_k(\wK,\iota)$
  is hyperbolic, hence the condition that $(\wB,\wtau)$ is hyperbolic
  implies $(A_0,\sigma_0)\otimes_k(\wK,\iota)$ is hyperbolic. But
  $(A_0,\sigma_0)\otimes_k(\wK,\iota)$ is anisotropic since
  (b)~$\Rightarrow$~(a), hence $A_0=\{0\}$ and therefore $(A,\sigma)$
  is hyperbolic.

  The proof of the equivalence of (c) and (d) (resp.\ (c') and (d'))
  is similar; we omit it.
\end{proof}

We next make some observations on the norm group $N(L/\wF)$ of a
finite-degree field extension $L$ of $\wF$. Recall that the $x$-adic
valuation on $\wF$ extends uniquely to a valuation on $L$. We let $v$
denote this valuation. Let $\ell$ be the residue field of $L$ and $M$
be the unramified closure of $\wF$ in $L$, which is the unique
unramified extension of $\wF$ in $L$ with residue field $\ell$ (see
\cite[Prop.~A.17]{TW15}). Let also $\pi$ be a uniformizer of $L$, and
let
\[
[L:M]=e \qquad\text{and}\qquad [M:\wF]=f.
\]
By \cite[Ch.~II, Th.~2]{Ser68}, we may identify $L=\ell((\pi))$ and
$M=\ell((x))$ since the characteristic of $k$ is zero. If $e=1$, we
take $\pi=x$. If $e>1$, let
\[
u=N_{L/M}(\pi)x^{-1}\in M^\times.
\]
Since $v\bigl(N_{L/M}(\pi)\bigr)=e\,v(\pi)=v(x)$, it follows that
$v(u)=0$. We may therefore consider the residue
$\overline{u}\in\ell^\times\subset M^\times$.

\begin{lemma}
  \label{lem:norms}
  \begin{enumerate}
  \item[(a)]
    If $e=1$ and $f$ is even, then $N(L/\wF)\subset N(\wK/\wF)\cdot
    N(\ell/k)$.
  \item[(b)]
    If $e$ is even, then $N(L/\wF)\subset
  N(\wK/\wF)\cdot\{1,N_{\ell/k}(-\overline{u})\}$ and
  $x\equiv-\overline u \bmod L^{\times2}$.
  \end{enumerate}
\end{lemma}

\begin{proof}
  (a)
  Every nonzero
  element of 
  $L$ can be written in the form $ax^r(1+m)$ for some
  $a\in\ell^\times$, some $r\in\mathbb{Z}$ and some $m\in L$ such that
  $v(m)>0$. We have
  \[
  N_{L/\wF}\bigl(ax^r(1+m)\bigr) = N_{\ell/k}(a)\, x^{rf}
  N_{L/\wF}(1+m).
  \]
  Since $f$ is even and $N_{\wK/\wF}(\xi)=-x$ it follows that
  \[
  x^{rf}=N_{\wK/\wF}(\xi^{rf})\in N(\wK/\wF).
  \]
  Moreover, Hensel's lemma shows that $1+m\in L^{\times2}$, hence
  $N_{L/\wF}(1+m)\in\wF^{\times2}\subset N(\wK/\wF)$. Therefore, the
  norm of every nonzero element in $L$ lies in $N(\wK/\wF)\cdot
  N(\ell/k)$. 
\smallbreak\par\noindent
  (b)
  In this case $L$ and $M$ have the same residue field, hence every
  element $y_0\in L$ such that 
  $v(y_0)=0$ can be written as $y_0=z(1+m)$ for some $z\in M^\times$
  and some $m\in L$ with $v(m)>0$. Therefore, for every element $y\in
  L^\times$ there exist $z\in M^\times$, $m\in L$ with $v(m)>0$ and
  $r\in\mathbb{Z}$ such that $y=z\pi^r(1+m)$. Then
  \begin{equation}
    \label{eq:N}
    N_{L/\wF}(y)=N_{M/\wF}(z)^eN_{M/\wF}(ux)^rN_{L/\wF}(1+m).
  \end{equation}
  Since $e$ is even, $N_{M/\wF}(z)^e\in\wF^{\times2}$. Similarly, 
  $N_{L/\wF}(1+m)\in\wF^{\times2}$ because $1+m\in
  L^{\times2}$ by Hensel's lemma. Moreover,
  $N_{M/\wF}(ux)=N_{M/\wF}(-u)\,(-x)^f=N_{M/\wF}(-u)N_{\wK/\wF}(\xi)^f$, hence
  from~\eqref{eq:N} it follows that
  \[
  N_{L/\wF}(y)\in N_{M/\wF}(ux)^r\cdot\wF^{\times2} \subset
  N_{M/\wF}(-u)^r\cdot  N(\wK/\wF).
  \]
  Since $\overline{u\overline{u}^{-1}}=1$, Hensel's lemma shows that
  $u\overline{u}^{-1}\in M^{\times2}$, hence $N_{M/\wF}(u)\equiv
  N_{\ell/k}(\overline u)\bmod\wF^{\times2}$. The first statement
  in~(b) is thus proved.

  To prove the second part, consider the minimal polynomial of $\pi$
  over $M$:
  \[
  X^e-a_1X^{e-1}+a_2X^{e-2}-\cdots+a_e\in M[X].
  \]
  Each coefficient $a_i$ is a sum of products of $i$ conjugates of
  $\pi$ in an algebraic closure of $L$, hence $v(a_i)\geq
  i\,v(\pi)$. But $a_i\in M$ and $v(M^\times) = e\,v(\pi)\mathbb{Z}$,
  hence in fact $v(a_i)\geq e\,v(\pi)$. Moreover, $v(a_e)=e\,v(\pi)$
  because $a_e=N_{L/M}(\pi)$, hence $v(a_ia_e^{-1})\geq0$ for $i=1$,
  \ldots, $e$. Therefore, taking residues in
  the equation
  \[
  \frac{\pi^e}{a_e} - \frac{a_1}{a_e}\pi^{e-1} +
  \frac{a_2}{a_e}\pi^{e-2} - \cdots+1=0,
  \]
  we obtain $\overline{\bigl(\frac{\pi^e}{a_e}\bigr)}=-1$. Note that
  $a_e=N_{L/M}(\pi)=xu$, hence
  $\overline{\bigl(\frac{\pi^e}{-xu}\bigr)}=1$. By Hensel's lemma
  again, it follows that $\frac{\pi^e}{-xu}\in L^{\times2}$, hence
  $-xu\in L^{\times2}$ because $e$ is even. Since $u\equiv\overline u
  \bmod M^{\times2}$, we finally get $x\equiv -\overline u\bmod
  L^{\times2}$. 
\end{proof}

We now turn to the problem mentioned at the beginning of this section,
which is to compute $\PGU(\wB,\wtau)/R$ in terms of $(A,\sigma)$. In
view of Proposition~\ref{prop:isohyp}, we assume $(A,\sigma)$ and
$(\wB,\wtau)$ are not hyperbolic for the rest of this section.

\begin{lemma}
  \label{lem:G}
  $
  G(\wB,\wtau)=N(\wK/\wF)\cdot G(A,\sigma).
  $
\end{lemma}

\begin{proof}
  Consider a Witt decomposition
  $(A,\sigma)\simeq(A_0,\sigma_0)\boxplus(A_1,\sigma_1)$ with
  $\sigma_0$ anisotropic and $\sigma_1$ hyperbolic. Then
  $(A_1,\sigma_1)\otimes_k(\wK,\iota)$ 
  is hyperbolic, hence
  \[
  G(\wB,\wtau)=G\bigl((A_0,\sigma_0)\otimes_k(\wK,\iota)\bigr)
  \quad\text{and similarly}\quad
  G(A,\sigma)=G(A_0,\sigma_0).
  \]
  Therefore, substituting $(A_0,\sigma_0)$ for $(A,\sigma)$ we may
  assume $\sigma$ is anisotropic.

  Let $g=\sum_{i=r}^\infty a_i\xi^i\in \GU(\wB,\wtau)$, with $a_i\in A$
  for all $i$, and $a_r\neq0$. Because $\wtau(g)g\in \wF^\times$ and
  $\sigma$ is anisotropic, we have $\sigma(a_r)a_r\in k^\times$. Then
  $x^{-r}(\sigma(a_r)a_r)^{-1}\wtau(g)g\in k[[x]]$, and
  \[
  \wtau(g)g=\sigma(a_r)a_r(-x)^r(1+m) \qquad\text{for some $m\in
    x\,k[[x]]$.}
  \]
  Hensel's lemma yields $1+m\in\wF^{\times2}$, hence $(-x)^r(1+m)\in
  N(\wK/\wF)$. Since $\sigma(a_r)a_r\in G(A,\sigma)$, it follows that
  $G(\wB,\wtau)\subset N(\wK/\wF)\cdot G(A,\sigma)$. The reverse
  inclusion is clear.
\end{proof}

We next consider $\Hyp(\wB,\wtau)$.

\begin{lemma}
  \label{lem:Hyp}
  Let $L$ be a finite-degree field extension of $\wF$ such that
  $(\wB,\wtau)_L$ is hyperbolic, and let $\ell$ be the residue field
  of $L$. The following properties hold:
  \begin{enumerate}
  \item[(a)]
  $[L:\wF]$ is even.
  \item[(b)]
  If $L$ is unramified, then $N(L/\wF)\subset
  N(\wK/\wF)\cdot\Hyp(A,\sigma)$. 
  \item[(c)]
  If $N(L/\wF)\not\subset N(\wK/\wF)\cdot\Hyp(A,\sigma)$, then there
  exist $\lambda\in\ell^\times$ and $g\in A_\ell$ such that
  $\sigma_\ell(g)=g$, $g^2=\lambda$, and
  $
  N(L/\wF)\subset N(\wK/\wF)\cdot\{1,N_{\ell/k}(\lambda)\}.
  $
  \end{enumerate}
\end{lemma}

\begin{proof}
  (a) Since $(\wB,\wtau)$ is not hyperbolic, it follows from a theorem
  of Bayer-Fluckiger and Lenstra \cite[Prop.~1.2]{BL90} that
  $(\wB,\wtau)$ remains non-hyperbolic over every odd-degree extension
  of $\wF$. 
\smallbreak\par\noindent
  (b)
  If $L$ is unramified, then we may identify $L=\ell((x))$. By
  applying Proposition~\ref{prop:isohyp} after extending scalars of
  $A$ from $k$ to $\ell$, we see that $(A,\sigma)_\ell$ is
  hyperbolic. Therefore, $N(\ell/k)\subset\Hyp(A,\sigma)$, and by
  Lemma~\ref{lem:norms}(a) it follows that $N(L/\wF)\subset
  N(\wK/\wF)\cdot\Hyp(A,\sigma)$. 
\smallbreak\par\noindent
  (c)
  Let $M\subset L$ be the unramified closure of $\wF$ in $L$. If
  $(\wB,\wtau)_M$ is hyperbolic, then~(b) yields $N(M/\wF)\subset
  N(\wK/\wF)\cdot\Hyp(A,\sigma)$. But $N(L/\wF)\subset N(M/\wF)$,
  hence this case does not arise when $N(L/\wF)\not\subset
  N(\wK/\wF)\cdot\Hyp(A,\sigma)$. Therefore, the hypothesis implies
  $(\wB,\wtau)_M$ is not hyperbolic. From the theorem of
  Bayer-Fluckiger and Lenstra mentioned in~(a), it follows that
  $[L:M]$ is even, hence we may apply Lemma~\ref{lem:norms}(b) to
  obtain (with the notation of that lemma)
  \begin{equation}
  \label{eq:u}
  N(L/\wF)\subset N(\wK/\wF)\cdot\{1,N_{\ell/k}(-\overline u)\}.
  \end{equation}
  To complete the proof, we show that $\lambda=-\overline u$ satisfies
  the requirements.

  First, note that $-\overline u\notin \ell^{\times2}$ since otherwise
  \eqref{eq:u} yields $N(L/\wF)\subset
  N(\wK/\wF)$. Lemma~\ref{lem:norms}(b) 
  shows that $x\equiv-\overline u \bmod L^{\times2}$, hence
  \[
  \wK L\simeq L(\sqrt{-\overline u}) \simeq \ell(\sqrt{-\overline
    u})((\pi)).
  \]
  After scalar extension to $L$, the automorphism $\iota$ of $\wK$
  yields the nontrivial automorphism $\alpha$ of
  $\ell(\sqrt{-\overline u})((\pi))$ over $\ell((\pi))$. Therefore,
  \[
  (\wB,\wtau)_L\simeq (A,\sigma)_\ell \otimes_\ell
  (\ell(\sqrt{-\overline u}),\alpha) \otimes_\ell \ell((\pi)).
  \]
  Since $(\wB,\wtau)_L$ is hyperbolic, it follows from the equivalence
  of (c') and (d') in Proposition~\ref{prop:isohyp} that
  $(A,\sigma)_\ell \otimes_\ell(\ell(\sqrt{-\overline u}),\alpha)$ is
  hyperbolic. Note that we are 
  not in the exceptional case of Theorem~\ref{thm:hypcond}, for if
  $A_\ell$ is split and $\sigma_\ell$ is symplectic then
  $(A,\sigma)_\ell$ is hyperbolic, hence $(\wB,\wtau)_M$ is
  hyperbolic. Therefore, Theorem~\ref{thm:hypcond} yields an element
  $g\in A_\ell$ such that $\sigma_\ell(g)=g$ and $g^2=-\overline u$,
  which completes the proof.
\end{proof}

In order to account for case~(c) of Lemma~\ref{lem:Hyp}, we introduce
the following group $S(A,\sigma)$:
\begin{dfn}
  $S(A,\sigma)\subset k^\times$ is the subgroup generated by the
  elements $N_{\ell/k}(\lambda)$, where $\ell$ is a finite-degree
  field extension of $k$ and $\lambda\in\ell^\times$ is such that
  there exists $g\in A_\ell$ satisfying $\sigma_\ell(g)=g$ and
  $g^2=\lambda$. 
\end{dfn}

Note that the conditions $\sigma_\ell(g)=g$ and $g^2=\lambda$ imply
$\sigma_\ell(g)g=\lambda$, hence $\lambda\in
G(A_\ell,\sigma_\ell)$. By \cite[(12.21)]{KMRT98} it follows that
$N_{\ell/k}(\lambda)\in G(A,\sigma)$, hence $S(A,\sigma)\subset
G(A,\sigma)$. 

\begin{theorem}
  \label{thm:PGU/R}
  There is a canonical group isomorphism
  \[
  G(\wB,\wtau)\big/\bigl(N(\wK/\wF)\cdot \Hyp(\wB,\wtau)\bigr) \simeq
  G(A,\sigma)\big/\bigl(k^{\times2}\cdot\Hyp(A,\sigma)\cdot
  S(A,\sigma)\bigr). 
  \]
\end{theorem}

\begin{proof}
  Lemma~\ref{lem:Hyp} shows that $N(\wK/\wF)\cdot \Hyp(\wB,\wtau)
  \subset N(\wK/\wF)\cdot \Hyp(A,\sigma)\cdot S(A,\sigma)$. We show
  that this inclusion is an equality.

  Every field extension that makes $(A,\sigma)$ hyperbolic also makes
  $(\wB,\wtau)$ hyperbolic, hence 
  \[
  \Hyp(A,\sigma)\subset\Hyp(\wB,\wtau).
  \]
  Now, assume $\ell$ is a field extension of $k$ of
  finite degree~$f$, and $\lambda\in\ell^\times$, $g\in A_\ell$
  satisfy $\sigma_\ell(g)=g$ and $g^2=\lambda$. If
  $\lambda\in\ell^{\times2}$, then $N_{\ell/k}(\lambda)\in
  k^{\times2}\subset N(\wK/\wF)$. If $\lambda\notin\ell^{\times2}$, then
  Theorem~\ref{thm:hypcond} shows that
  $(A,\sigma)_\ell\otimes(\ell(\sqrt\lambda),\alpha)$ is hyperbolic,
  where $\alpha$ is the nontrivial automorphism of
  $\ell(\sqrt\lambda)$ over $\ell$. Let $L=\ell((\pi))$ where
  $\pi^2=\lambda x$. Then
  \[
  (A,\sigma)_\ell\otimes(\ell(\sqrt\lambda),\alpha)\otimes_\ell L
  \simeq (\wB,\wtau)_L,
  \]
  hence $(\wB,\wtau)_L$ is hyperbolic. Moreover, $N_{L/\wF}(\pi)=
  N_{\ell((x))/\wF}(-\lambda x) = (-x)^f N_{\ell/k}(\lambda)$. Since
  $N_{\wK/\wF}(\xi)=-x$, it follows that $N_{\ell/k}(\lambda)\in
  N(\wK/\wF)\cdot N(L/\wF) \subset N(\wK/\wF)\cdot\Hyp(\wB,\wtau)$. We
  have thus 
  shown $S(A,\sigma)\subset N(\wK/\wF)\cdot \Hyp(\wB,\wtau)$, hence
  \[
  N(\wK/\wF)\cdot \Hyp(\wB,\wtau) = N(\wK/\wF)\cdot \Hyp(A,\sigma)
  \cdot S(A,\sigma).
  \]
  On the other hand, Lemma~\ref{lem:G} yields
  $G(\wB,\wtau)=N(\wK/\wF)\cdot G(A,\sigma)$. Since $N(\wK/\wF)\cap
  k^\times = k^{\times2}$, it follows that the inclusion
  $G(A,\sigma)\subset 
  N(\wK/\wF)\cdot G(A,\sigma)$ induces an isomorphism
  \[
  \frac{G(A,\sigma)}{k^{\times2}\cdot\Hyp(A,\sigma)\cdot S(A,\sigma)}
  \stackrel\sim\to  \frac{N(\wK/\wF)\cdot G(A,\sigma)}{N(\wK/\wF)\cdot
    \Hyp(A,\sigma)\cdot S(A,\sigma)} =
  \frac{G(\wB,\wtau)}{N(\wK/\wF)\cdot\Hyp(\wB,\wtau)}.
  \qedhere
  \]
\end{proof}

We conclude this section with two special cases:

\begin{corollary}
  \label{cor:symp}
  Suppose $\deg A\equiv2\bmod4$ and $\sigma$ is symplectic. Then
  \[
  S(A,\sigma)\subset k^{\times2}\qquad\text{and}\qquad 
  \PGU(\wB,\wtau)/R=1.
  \]
\end{corollary}

\begin{proof}
  Let $\ell$ be a finite-degree field extension of $k$ and
  $\lambda\in\ell^\times$ be such that $g^2=\lambda$ for some
  $\sigma_\ell$-symmetric element $g\in A_\ell$. Since $\sigma_\ell$
  is symplectic and $\deg A_\ell\equiv2\bmod4$, the reduced Pfaffian
  characteristic polynomial of $g$ (see~\cite[(2.9)]{KMRT98}) has odd
  degree, hence $\lambda\in\ell^{\times2}$. Therefore,
  $S(A,\sigma)\subset k^{\times2}$, and Theorem~\ref{thm:PGU/R} yields
  \[
  G(\wB,\wtau)\big/\bigl(N(\wK/\wF)\cdot\Hyp(\wB,\wtau)\bigr) \simeq
  G(A,\sigma)\big/\bigl(k^{\times2}\cdot\Hyp(A,\sigma)\bigr).
  \]
  The right side is trivial because Merkurjev has shown
  \cite[Prop.~4]{Mer96} that $\gPGSp(A,\sigma)$ is stably
  rational. By~\eqref{eq:MerkPGU/R}, it follows that
  $\PGU(\wB,\wtau)/R=1$. 
\end{proof}

Recall that when the involution $\sigma$ is orthogonal, we let
$G^+(A,\sigma)$ denote the group of multipliers of proper similitudes,
and $G^-(A,\sigma)$ the coset of multipliers of improper similitudes
(if any).

\begin{lemma}
  \label{lem:normult}
  Let $\ell$ be a finite-degree field extension of $k$. Then
  $N_{\ell/k}\bigl(G^+(A_\ell,\sigma_\ell)\bigr) \subset
  G^+(A,\sigma)$. Moreover,
  \[
  N_{\ell/k}\bigl(G^-(A_\ell,\sigma_\ell)\bigr) \subset
  \begin{cases}
    G^+(A,\sigma)&\text{if $[\ell:k]$ is even,}\\
    G^-(A,\sigma)&\text{if $[\ell:k]$ is odd.}
  \end{cases}
  \]
\end{lemma}

\begin{proof}
  Let $\mu\in G(A_\ell,\sigma_\ell)$ and let $\delta\in k^\times$ be a
  representative of the square class $\disc\sigma$. Recall from
  \cite[Th.~A]{MT95} or \cite[(13.38)]{KMRT98} that
  \[
  \mu\in G^+(A_\ell,\sigma_\ell) \text{ if and only if }
  [(\delta,\mu)_\ell]=0 \text{ in $\Br(\ell)$,}
  \]
  and
  \[
  \mu\in G^-(A_\ell,\sigma_\ell) \text{ if and only if }
  [(\delta,\mu)_\ell] = [A_\ell] \text{ in $\Br(\ell)$.}
  \]
  Taking the corestriction from $\ell$ to $k$, we obtain
  $[(\delta,N_{\ell/k}(\mu))_k] = 0$ if $\mu\in
  G^+(A_\ell,\sigma_\ell)$, and $[(\delta,N_{\ell/k}(\mu))_k] =
  [\ell:k]\cdot [A]$ if $\mu\in
  G^-(A_\ell,\sigma_\ell)$. By~\cite[(12.21)]{KMRT98} we already know
  $N_{\ell/k}(\mu)\in G(A,\sigma)$; the lemma follows.
\end{proof}

\begin{corollary}
  \label{cor:orth}
  Suppose $\deg A\equiv0\bmod4$ and $\sigma$ is orthogonal. Then
  $S(A,\sigma)\subset G^+(A,\sigma)$ and there is a canonical
  surjective map $\varphi\colon\PGU(\wB,\wtau)/R\to
  G(A,\sigma)/G^+(A,\sigma)$. 
\end{corollary}

\begin{proof}
  Let $\ell$ be a finite-degree field extension of $k$ and
  $\lambda\in\ell^\times$ be such that $g^2=\lambda$ for some
  $\sigma_\ell$-symmetric $g\in A_\ell$. Then
  $\lambda=\sigma_\ell(g)g$ and
  $\Nrd_{A_\ell}(g)=(-\lambda)^{\frac12\deg A}$. Since $\frac12\deg A$
  is even, it follows that $g$ is a proper similitude, hence
  $\lambda\in G^+(A_\ell,\sigma_\ell)$. Lemma~\ref{lem:normult} then
  yields $N_{\ell/k}(\lambda)\in G^+(A,\sigma)$, hence
  $S(A,\sigma)\subset G^+(A,\sigma)$. On the other hand, we have
  $\Hyp(A,\sigma)\subset G^+(A,\sigma)$ (see \eqref{eq:MerkPG0/R}),
  hence there is a canonical surjective map
  \[
  G(A,\sigma)\big/\bigl(k^{\times2}\cdot \Hyp(A,\sigma)\cdot
  S(A,\sigma)\bigr) \to G(A,\sigma)/G^+(A,\sigma).
  \] 
  The corollary follows from Theorem~\ref{thm:PGU/R}.
\end{proof}

\begin{remark}
  \label{rem:BMT}
  The cohomological invariant
  $\theta_2\colon\gPGU(\wB,\wtau) \to H^3(\bullet,\mu_2)$ defined in
  \cite[Prop.~11]{BMT04} yields a map
  $\theta_{2\widehat F}\colon \PGU(\wB,\wtau)\to H^3(\widehat
  F,\mu_2)$ that factors through $\varphi$. If
  $G^-(A,\sigma)\neq\varnothing$, its image is $\{0, (x)\cup[A]\}$,
  where $(x)\in H^1(\widehat F,\mu_2)$ is the square class of
  $x\in\widehat F^\times$, see \cite[Prop.~13]{BMT04}. Therefore, the
  map $\varphi$ is nontrivial if and only if the map
  $\PGU(\wB,\wtau)/R\to H^3(\widehat F,\mu_2)$ induced by
  $\theta_{2\widehat F}$ is nontrivial.
\end{remark}

\section{Examples of non-rational adjoint groups of type
  $^2\mathsf{A}_n$} 
\label{sec:examples}

\subsection{Case $n\equiv3\bmod4$}
\label{subsec:3mod4}

In this subsection, we prove Theorem~\ref{thm:A3}. Let $(A,\sigma)$ be
a central simple algebra with orthogonal involution of degree multiple
of $4$ over a field $k$ of characteristic~$0$, and let $F=k(x)$ be the
rational function field in one indeterminate over $k$. We let
$K=F(\sqrt x)$, write $\iota$ for the nontrivial automorphism of
$K/F$, and let
\[
(B,\tau)=(A,\sigma)\otimes_k(K,\iota).
\]
Thus, $(B,\tau)$ is a generic unitary extension of $(A,\sigma)$. We
assume $A$ is not split and $\disc\sigma$ is not
trivial. Proposition~\ref{prop:improper} yields a finitely generated
extension $k_1=k(y_1,\ldots,y_r)$ of $k$ such that $A_{k_1}$ is not
split and $(A,\sigma)_{k_1}$ admits improper similitudes. We may
assume $k_1$ and $K$ both lie in some field extension of $k$ and are
linearly disjoint over $k$, so we may consider the composite field
extensions
\[
F_1=k_1(x)=F(y_1,\ldots,y_r) \subset K_1=F_1(\sqrt x) =
K(y_1,\ldots,y_r).
\]
Let also $\widehat F_1=k_1((x))$ and $\widehat K_1=\widehat F_1(\sqrt
x)$. Because $\deg A\equiv0\bmod4$, Corollary~\ref{cor:orth} yields a
surjective map
\[
\PGU(B_{\widehat F_1},\tau_{\widehat F_1})/R \to
G(A_{k_1},\sigma_{k_1})/G^+(A_{k_1},\sigma_{k_1}).
\]
Since $A_{k_1}$ is not split and $(A,\sigma)_{k_1}$ admits improper
similitudes, the right side is not trivial
(see~\eqref{eq:propsimnothing}). We have thus found an 
extension $\widehat F_1$ of $F$ such that $\PGU(B_{\widehat F_1},
\tau_{\widehat F_1})/R\neq1$, which means that $\gPGU(B,\tau)$ is not
$R$-trivial. The proof of Theorem~\ref{thm:A3} is thus complete.

\subsection{Case $n\equiv1\bmod4$}

We start with the following construction, which will be iterated in
the sequel: $Q$ is a central quaternion division
algebra over an arbitrary field $E$ of characteristic zero. Let
$(V,h)$, $(V',h')$ be nondegenerate skew-hermitian spaces over
$Q$ (with respect to the conjugation involution on $Q$). Consider the
field of Laurent series in one indeterminate over $E$,
\[
\widehat E=E((t)),
\]
and let $\widehat Q=Q_{\widehat E}$, $(\widehat V,
\widehat h)=(V,h)_{\widehat E}$, $(\widehat V', \widehat h') =
(V',h')_{\widehat E}$ be the division algebra and skew-hermitian
spaces obtained by extending scalars from $E$ to $\widehat E$. We may
then form the following nondegenerate skew-hermitian space
over $\widehat Q$:
\[
(W,h_W)=(\widehat V\oplus \widehat V',\; \widehat h\perp \langle t
\rangle \widehat h').
\]

\begin{proposition}
  \label{prop:sim1}
  \begin{enumerate}
  \item[(1)]
  If $h$ and $h'$ are anisotropic, then $h_W$ is
  anisotropic.
  \item[(2)]
  Assume $h$ and $h'$ are anisotropic and not similar. If there
  exists $\widehat g\in 
  \End_{\widehat Q}W$ such that $\ad_{h_W}(\widehat g)=\widehat g$ and
  $\widehat g^2=\lambda$ for
  some $\lambda\in \widehat E^\times$, then there exist $\lambda_0\in
  E^\times$ and $g\in\End_QV$, $g'\in\End_QV'$ such that
  $\lambda\equiv\lambda_0\bmod \widehat E^{\times2}$,
  $\ad_{h}(g)=g$, $\ad_{h'}(g')=g'$, and $g^2={g'}^2=\lambda_0$.
  \end{enumerate}
\end{proposition}

\begin{proof}
  Throughout the proof, we assume $h$ and $h'$ are anisotropic.
  Let $v$ be the $t$-adic valuation on $\widehat E$. We also write $v$
  for the valuation on $\widehat Q$ extending $v$. Observe that every
  vector $x\in \widehat V$ can be written as a 
  series $x=\sum_{i=r}^\infty x_it^i$ with $x_i\in V$ for all
  $i$. For such nonzero $x$, define $\nu(x)=\inf\{i\mid x_i\neq0\}$,
  and let $\nu(0)=\infty$. Similarly, for $x'=\sum_{j=s}^\infty
  x'_jt^j\in \widehat V'$, let $\nu'(x')=\inf\{j\mid x'_j\neq0\}$, and
  $\nu'(0)=\infty$. Finally, for $x\in\widehat V$ 
  and $x'\in\widehat V'$, let
  \[
  \nu_W(x+x')=\min\bigl\{\nu(x),{\textstyle\frac12}+\nu'(x')\bigr\}
  \in 
  {\textstyle\frac12}\mathbb{Z}\cup\{\infty\}.
  \]
  The map $\nu_W\colon W\to \frac12\mathbb{Z}\cup\{\infty\}$ is a
  $v$-norm on $W$ (see \cite[p.~83]{TW15}). Since $h$ and $h'$ are
  anisotropic, it follows that $\nu(x)=\frac12
  v\bigl(\widehat h(x,x)\bigr)$ for all $x\in\widehat V$ and
  $\nu'(x')=\frac12 v\bigl(\widehat h'(x',x')\bigr)$ for all
  $x'\in\widehat V'$. Therefore,
  \begin{equation}
  \label{eq:norm}
  \nu_W(w)={\textstyle\frac12}v\bigl(h_W(w,w)\bigr) \qquad\text{for
    all $w\in W$.}
  \end{equation}
  It follows that $h_W$ is anisotropic, proving~(1). Moreover, it is
  easy to see that 
  \begin{equation}
    \label{eq:compat}
    v\bigl(h_W(w_1,w_2)\bigr)\geq \nu_W(w_1)+\nu_W(w_2)
    \qquad\text{for all $w_1$, $w_2\in W$}.
  \end{equation}

  To prove~(2), let $\widehat g\in\End_{\widehat Q}W$ be such that
  $\ad_{h_W}(\widehat g)=\widehat g$ and $\widehat g^2=\lambda$.
  We then have $\lambda\in G(\End_{\widehat Q}W, \ad_{h_W})$, hence
  $v(\lambda)\in2\mathbb{Z}$ by \cite[Prop.~2.3]{QMT2} because $h$
  and $h'$ are not similar. We may then find $\lambda_0\in
  E^\times$, $r\in\mathbb{Z}$, and $m\in \widehat E$ with $v(m)>0$
  such that 
  $\lambda=\lambda_0t^{2r}(1+m)$. Hensel's lemma yields $1+m\in
  \widehat E^{\times2}$, hence $\lambda\equiv\lambda_0\bmod
  \widehat E^{\times2}$. If 
  $\lambda_1\in \widehat E^\times$ is such that
  $\lambda=\lambda_0\lambda_1^2$, 
  then substituting $\widehat g\lambda_1^{-1}$ for $\widehat g$ we may
  (and will) assume for the rest of the proof that $\widehat
  g^2=\lambda_0\in E^\times$. 

  The remaining claims can be established by the graded algebra
  technique of \cite[Sec.~3.2]{QMT1}. For the convenience of the
  reader, we give 
  an alternative more elementary argument.

  Since $\widehat g\in\End_{\widehat Q}W$ satisfies
  $\ad_{h_W}(\widehat g)=\widehat g$ and
  $\widehat g^2=\lambda_0\in E^\times$, it follows that
  \[
  h_W\bigl(\widehat g(w),\widehat g(w)\bigr)=\lambda_0h_W(w,w)
  \qquad\text{for all $w\in W$}, 
  \]
  hence by \eqref{eq:norm} $\nu_W\bigl(\widehat g(w)\bigr)=\nu_W(w)$
  for all $w\in W$. In particular, for $x\in V$ we have
  $\nu_W\bigl(\widehat g(x)\bigr)=0$
  unless $x=0$, hence there exist $g(x)\in V$ and $g_+(x)\in W$
  such that
  \[
  \widehat g(x)=g(x)+g_+(x) \quad\text{and}\quad
  \nu_W\bigl(g_+(x)\bigr)>0. 
  \]
  The elements $g(x)$ and $g_+(x)$ are uniquely determined by these
  conditions, and the map $g\colon V\to V$ is
  $Q$-linear. Similarly, for $x'\in V'$ we have
  $\nu_W\bigl(\widehat g(x')\bigr)=\frac12$ if $x'\neq0$, and we get
  uniquely determined elements $g'(x')\in V'$, $g'_+(x')\in W$ such
  that 
  \[
  \widehat g(x')=g'(x')+g'_+(x') \quad\text{and}\quad
  \nu_W\bigl(g'_+(x')\bigr) >  {\textstyle\frac12}.
  \]
  The map $g'$ belongs to $\End_QV'$.
  For $x_1$, $x_2\in V$ we have
  \[
  h_W(\widehat g(x_1),x_2) = h(g(x_1),x_2) + h_W(g_+(x_1),x_2)
  \]
  and \eqref{eq:compat} shows that
  $v\bigl(h_W(g_+(x_1),x_2)\bigr)>0$. Therefore, letting
  $\cO_{\widehat Q}$ denote the valuation ring of $\widehat Q$ and
  $\invo\colon\cO_{\widehat Q}\to Q$ the residue map, we have for all
  $x_1$, $x_2\in V$
  \[
  h(g(x_1),x_2) = \overline{h_W(\widehat g(x_1),x_2)}
  \quad\text{and similarly}\quad 
  h\bigl(x_1,g(x_2)\bigr) = \overline{h_W\bigl(x_1,\widehat
    g(x_2)\bigr)}.
  \]
  Since $\ad_{h_W}(\widehat g)=\widehat g$ and $\widehat
  g^2=\lambda_0$, it follows that 
  $\ad_{h}(g) = g$ and $g^2=\lambda_0$.

  Likewise, for $x'_1$, $x'_2\in V'$ we have
  \[
  h_W(\widehat g(x'_1),x'_2) =t\,h'(g'(x'_1),x'_2) +
                              h_W(g'_+(x'_1),x'_2)
  \]
  and \eqref{eq:compat} yields $v\bigl(h_W(g'_+(x'_1),x'_2)\bigr)>1$,
  hence for $x'_1$, $x'_2\in V'$, 
  \[
  h'(g'(x'_1),x'_2) = \overline{t^{-1}h_W(\widehat g(x'_1),x'_2)}
  \quad\text{and similarly}\quad 
  h'\bigl(x'_1,g'(x'_2)\bigr) =
  \overline{t^{-1}h_W\bigl(x'_1,g_W(x'_2)\bigr)}.
  \]
  Since $\ad_{h_W}(\widehat g)=\widehat g$ and $\widehat
  g^2=\lambda_0$, it follows that 
  $\ad_{h'}(g') = g'$ and ${g'}^2=\lambda_0$.
\end{proof}

For the application in Theorem~\ref{thm:main} below, we need to show
that Proposition~\ref{prop:sim1} still holds after an odd-degree
scalar 
extension. Let $L$ be an odd-degree field extension of
$\widehat E$. Extending scalars from $\widehat E$ to $L$, we obtain
the quaternion division algebra $\widehat Q_L$ over $L$, the $\widehat
Q_L$-vector space $W_L$ and the skew-hermitian form
$(h_W)_L$ on $W_L$. 

\begin{corollary}
  \label{cor:sim1}
  Assume
  $h$ and $h'$ are not similar and anisotropic. If there
  exists $\widehat g\in \End_{\widehat Q_L} W_L$ such that
  $\ad_{(h_W)_L}(\widehat g)=\widehat g$ and $\widehat g^2=\lambda$
  for some $\lambda\in L^\times$, then there exists an odd-degree field
  extension $L_0$ of $E$ contained in $L$, a scalar $\lambda_0\in
  L_0^\times$, and maps
  $g\in \End_{Q_{L_0}}(V)_{L_0}$,
  $g'\in\End_{Q_{L_0}}(V')_{L_0}$ such that
  $\lambda\equiv\lambda_0\bmod L^{\times2}$,
  $\ad_{h_{L_0}}(g)=g$, $\ad_{h'_{L_0}}(g')=g'$, and
  $g^2={g'}^2=\lambda_0$. 
\end{corollary}

\begin{proof}
  The $t$-adic valuation on $\widehat E$ extends uniquely to $L$
  because $\widehat E$ 
  is complete. Let $L_0$ be the residue field of $L$ and
  $\pi\in L$ be a uniformizer. Since $\cchar(E)=0$ we may identify
  $L=L_0((\pi))$, see \cite[Ch.~II, Th.~2]{Ser68}. Let
  $e=[L:L_0((t))]$ 
  and $f=[L_0:E]$ be the ramification index and residue
  degree. Since $ef=[L:\widehat E]$ is odd, both $e$ and $f$ are
  odd. We have $v(\pi^et^{-1})=0$, hence there exist $u\in L_0^\times$
  and $m\in L$ with $v(m)>0$ such that
  \[
  \pi^e=tu(1+m).
  \]
  Now, $1+m\in L^{\times2}$ by Hensel's lemma, and $e$ is odd, hence
  the last equation yields $t\equiv \pi u\bmod L^{\times2}$. Therefore,
  \[
  (h_W)_L \simeq \widehat h_L \perp \langle \pi u \rangle
  \widehat h'_L.
  \]
  Note $\widehat h_L =
  (h_{L_0})_{L_0((\pi))}$ and $\langle u\rangle
  \widehat h'_L = (\langle u\rangle
  h'_{L_0})_{L_0((\pi))}$. Since $f$ is odd, the anisotropic
  forms $h$ and $h'$ remain anisotropic under scalar extension to
  $L_0$ by a theorem of Parimala--Sridharan--Suresh
  \cite[Th.~3.5]{PSS01}. Moreover, since $h$ and $h'$ are not
  similar, a transfer argument due to Lewis \cite[Prop.~10]{Lew00}
  shows 
  that $h_{L_0}$ and $h'_{L_0}$ are not similar, and
  therefore $h_{L_0}$ and $\langle u\rangle h'_{L_0}$ are
  not similar either. Thus, we are in a position to apply
  Proposition~\ref{prop:sim1} with $L$ instead of $\widehat E$: if
  there exists $\widehat g\in\End_{\widehat Q_L}W_L$ as in the
  statement, then we may find $\lambda_0\in L_0^\times$ and $g$, $g'$ as
  required. (Note that $h'_{L_0}$ and $\langle u\rangle
  h'_{L_0}$ have the same adjoint involution.)
\end{proof}

Iterating the Laurent series construction, we apply
Corollary~\ref{cor:sim1} inductively to the following situation: let
$n\geq2$ be an arbitrary integer, let $Q$ be a central quaternion
division algebra over an arbitrary field $k_0$ of characteristic zero,
and let $q_1$, \ldots, $q_n\in Q$ be nonzero pure quaternions. Let
$a_i=q_i^2\in k_0^\times$. Consider the field of iterated Laurent
series in $n$ indeterminates
\[
k=k_0((t_1))\ldots((t_n))
\]
and the orthogonal involution $\sigma$ on $A=M_n(Q_k)$ adjoint to the
skew-hermitian form
\[
h=\langle t_1q_1,\ldots,t_nq_n\rangle.
\]

\begin{theorem}
  \label{thm:main}
  Suppose the skew-hermitian forms $\langle q_1\rangle$ and $\langle
  q_2\rangle$ are not similar. If $\ell$ is an odd-degree field
  extension of $k$ and $g\in A_\ell$ is such that $\sigma_\ell(g)=g$
  and $g^2=\lambda$ for some $\lambda\in \ell^\times\setminus
  \ell^{\times2}$, then there exists $\mu\in k_0^\times$ such that
  \begin{equation}
    \label{eq:commonslot}
    Q\simeq(a_1,\mu)_{k_0}\simeq\cdots\simeq (a_n,\mu)_{k_0}.
  \end{equation}
\end{theorem}

\begin{proof}
  For $i=1$, \ldots, $n$, let $E_i=k_0((t_1))\ldots((t_i))$. Consider
  the following skew-hermitian forms over $Q_{E_i}$:
  \[
  h_{i}=\langle t_1q_1,\ldots,t_{i}q_{i}\rangle
  \quad\text{and}\quad h'_{i}=\langle q_{i+1}\rangle \qquad\text{for
    $i=1$, \ldots, $n-1$.} 
  \]
  Let $V_i$ and $V'_i$ be the $Q_{E_i}$-vector spaces underlying $h_i$
  and $h'_i$ respectively.
  The form $h'_{i}$ is clearly anisotropic, and
  Proposition~\ref{prop:sim1}(1) applied inductively shows that
  $h_{i}$ is anisotropic. If $i=1$ the forms $h_{i}$ and $h'_{i}$
  are not similar by hypothesis; if $i\geq2$ they are not similar
  because they do not have the same dimension. 

  Suppose $\ell$ is an odd-degree field extension of $k=E_n$ and $g\in
  A_\ell$ is as in the statement of the
  theorem. Corollary~\ref{cor:sim1} (with $L=\ell$, $E=E_{n-1}$, and
  $\widehat E= E_n$) yields an odd-degree field extension $\ell_{n-1}$
  of $E_{n-1}$ contained in $\ell$, a scalar
  $\lambda_{n-1}\in\ell_{n-1}^\times$ such that
  $\lambda_{n-1}\equiv\lambda\bmod \ell^{\times2}$ and maps
  $g_{n-1}\in\End_{Q_{\ell_{n-1}}}(V_{n-1})_{\ell_{n-1}}$,
  $g'_{n-1}\in\End_{Q_{\ell_{n-1}}}(V'_{n-1})_{\ell_{n-1}}$, symmetric
  under $\ad(h_{n-1})$ and $\ad(h'_{n-1})$ respectively, such that
  $g_{n-1}^2={g'}_{n-1}^2=\lambda_{n-1}$. Applying again
  Corollary~\ref{cor:sim1} (with $L=\ell_{n-1}$ and $\widehat
  g=g_{n-1}$), we 
  obtain an odd-degree field extension $\ell_{n-2}$ of $E_{n-2}$
  contained in $\ell_{n-1}$, a scalar
  $\lambda_{n-2}\in\ell_{n-2}^\times$ such that
  $\lambda_{n-2}\equiv\lambda_{n-1}\bmod \ell_{n-1}^{\times2}$ and
  $g_{n-2}\in \End_{Q_{\ell_{n-2}}}(V_{n-2})_{\ell_{n-2}}$,
  $g'_{n-2}\in \End_{Q_{\ell_{n-2}}}(V'_{n-2})_{\ell_{n-2}}$,
  symmetric under $\ad(h_{n-2})$ and $\ad(h'_{n-2})$, such that
  $g_{n-2}^2={g'}_{n-2}^2=\lambda_{n-2}$. Repeating the procedure as
  many times as needed, we finally have field extensions
  \[
  \ell_1\subset \ell_2\subset\cdots\subset\ell_{n-1}
  \subset\ell_n=\ell,
  \]
  scalars $\lambda_i\in\ell_i^\times$ for $i=1$, \ldots, $n-1$ and
  $\lambda_n=\lambda$ such that
  \begin{equation}
    \label{eq:lambda}
    \lambda_i\equiv\lambda_{i+1} \bmod\ell_{i+1}^{\times2}
    \qquad\text{for $i=1$, \ldots, $n-1$,}
  \end{equation}
  and maps 
  \[
  g_1\in\End_{Q_{\ell_1}}(V_1)_{\ell_1},\quad 
  g'_{i}\in\End_{Q_{\ell_{i}}}(V'_{i})_{\ell_{i}}\qquad\text{for
    $i=1$, \ldots, $n-1$},
  \]
  each
  symmetric under the adjoint involution of the corresponding
  skew-hermitian form, such that
  \[
  g_1^2=\lambda_1 \quad\text{and}\quad {g'}_i^2=\lambda_i
  \quad\text{for $i=1$, \ldots, $n-1$}.
  \]
  Note that $V_1$, $V'_1$, $V'_2$, \ldots, $V'_{n-1}$ are
  $1$-dimensional, hence using bases we may identify
  $\End_{Q_{\ell_1}}(V_1)_{\ell_1}=Q_{\ell_1}$ and
  $\End_{Q_{\ell_i}}(V'_i)_{\ell_i}=Q_{\ell_i}$ for $i=1$, \ldots,
  $n-1$. From~\eqref{eq:lambda} it follows that
  $\lambda_i\equiv\lambda\mod\ell^{\times2}$, hence
  $\lambda_i\notin\ell_i^{\times2}$ for $i=1$, \ldots,
  $n-1$. Therefore, $g_1$, $g'_1$, \ldots, $g'_{n-1}$ are pure
  quaternions. The condition that $g_1$ is symmetric under $\ad(h_1)$
  then means that $g_1q_1=-q_1g_1$, hence $q_1$, $g_1$ are part of a
  quaternion base of $Q_{\ell_1}$ and therefore
  \[
  Q_{\ell_1}\simeq(a_1,\lambda_1)_{\ell_1}.
  \]
  Similarly, because $g'_i$ is symmetric under $\ad(h'_i)$ we have
  \[
  Q_{\ell_i}\simeq(a_{i+1},\lambda_i)_{\ell_i} \qquad\text{for $i=1$,
    \ldots, $n-1$.}
  \]
  Extending scalars to $\ell$ and using
  $\lambda_i\equiv\lambda\bmod\ell^{\times2}$ for $i=1$, \ldots,
  $n-1$, we see that
  \[
  Q_\ell\simeq (a_1,\lambda)_\ell\simeq\cdots\simeq
  (a_n,\lambda)_\ell.
  \]
  Taking the corestriction from $\ell$ to $k$, we obtain since
  $[\ell:k]$ is odd
  \begin{equation}
    \label{eq:fin}
    Q_k\simeq (a_1,N_{\ell/k}(\lambda))_k\simeq\cdots\simeq
    (a_n,N_{\ell/k}(\lambda))_k. 
  \end{equation}
  Recall from \cite[Ch.~VI, Cor.~1.3]{Lam05} that each element in
  $k^\times$ is in the 
  coset of some monomial $t_1^{\varepsilon_1}\ldots
  t_n^{\varepsilon_n}$ with each $\varepsilon_i\in\{0,1\}$ modulo
  $k_0^\times k^{\times2}$, hence we may find $\mu\in k_0^\times$ and
  $\varepsilon_1$, \ldots, $\varepsilon_n\in\{0,1\}$ such that
  \[
  N_{\ell/k}(\lambda)\equiv \mu t_1^{\varepsilon_1} \ldots
  t_n^{\varepsilon_n} \bmod k^{\times2}.
  \]
  But the Brauer class of $Q_k$ is unramified for the $(t_1,\ldots,
  t_n)$-adic valuation on $k$, hence we must have
  $\varepsilon_1=\cdots=\varepsilon_n=0$. From~\eqref{eq:fin}, it
  follows that $\mu$
  satisfies~\eqref{eq:commonslot}. 
\end{proof}

\begin{corollary}
  \label{cor:main}
  With the same hypotheses and notation as in Theorem~\ref{thm:main},
  suppose there does not exist any $\mu\in k_0^\times$
  satisfying~\eqref{eq:commonslot}. Then $S(A,\sigma)\subset
  G^+(A,\sigma)$. 
\end{corollary}

\begin{proof}
  If $\deg A\equiv0\bmod4$, the inclusion holds without any hypothesis
  on $Q$ by Corollary~\ref{cor:orth}. For the rest of the proof, we
  may thus assume $\deg A\equiv2\bmod4$, which means that $n$ is
  odd. Let $\ell$ be a finite-degree field extension of $k$ and
  $\lambda\in\ell^\times$ such that $g^2=\lambda$ for some
  $\sigma_\ell$-symmetric $g\in A_\ell$. As $\lambda=\sigma_\ell(g)g$
  and $\Nrd_{A_\ell}(g)=(-\lambda)^n$, it follows that $g$ is an
  improper similitude, hence $\lambda\in
  G^-(A_\ell,\sigma_\ell)$. Since $A_\ell$ is not split we must have
  $\lambda\notin\ell^{\times2}$, hence Theorem~\ref{thm:main} shows
  that $[\ell:k]$ is even because there is no $\mu\in k_0^\times$
  satisfying~\eqref{eq:commonslot}. Lemma~\ref{lem:normult} then
  yields $N_{\ell/k}(\lambda)\in G^+(A,\sigma)$.
\end{proof}

\begin{example}
  \label{ex:main}
  Let $k_0=k_*(a_1,a_2)$ be the field of rational functions in two
  indeterminates over an arbitrary field $k_*$ of characteristic
  zero. The quaternion algebra $Q=(a_1,a_2)_{k_0}$ contains pure
  quaternions $q_1$, $q_2$, $q_3$ satisfying
  \[
  q_1^2=a_1,\quad q_2^2=a_2,\quad
  q_3^2=a_1\bigl((1-a_1)^2(1+a_2)^2-4(1-a_1)a_2\bigr),
  \]
  see \cite[Ex.~3.12]{QMT1}. Let $a_3=q_3^2$. It is shown in
  \cite[Ex.~3.12]{QMT1} that there is no $\mu\in k_0^\times$ such that
  $Q\simeq(a_1,\mu)_{k_0}\simeq(a_2,\mu)_{k_0}\simeq(a_3,\mu)_{k_0}$. Note
  that the forms $\langle q_1\rangle$ and $\langle q_2\rangle$ are not
  similar since they do not have the same discriminant. Therefore, for
  arbitrary $n\geq3$ the construction before Theorem~\ref{thm:main}
  with $q_3=q_4=\cdots=q_n$ yields by Corollary~\ref{cor:main} an
  algebra with orthogonal involution $(A,\sigma)$ of degree~$2n$ such
  that $S(A,\sigma)\subset G^+(A,\sigma)$. For the completion
  $(\wB,\wtau)$ of the generic unitary
  extension $(B,\tau)$ as in \eqref{eq:Btau} we then have by
  \eqref{eq:MerkPGU/R} and Theorem~\ref{thm:PGU/R} a canonical
  surjective map $\varphi\colon \PGU(\wB,\wtau)/R\to
  G(A,\sigma)/G^+(A,\sigma)$ as in Corollary~\ref{cor:orth}. If $n$ is
  odd and $-1\in k_*^{\times2}$, it is shown in \cite[Cor.~3.13]{QMT1}
  that $(A,\sigma)$ admits improper similitudes. Since $A$ is not
  split it follows that $G(A,\sigma)\neq G^+(A,\sigma)$ (see
  \eqref{eq:propsimnothing}), hence 
  $\PGU(\wB,\wtau)/R\neq1$. Therefore, the group $\gPGU(B,\tau)$ is
  not $R$-trivial since $\gPGU(B,\tau)(\wF)=\PGU(\wB,\wtau)$. Note
  that the field of definition of $\gPGU(B,\tau)$ is the field
  $k_*(a_1,a_2,x)$ of rational functions in three variables over an
  arbitrary field of characteristic zero.
\end{example}



\begin{thebibliography}{KMRT98}
%
\bibitem[BL90]{BL90}
E.~Bayer-Fluckiger, H.~Lenstra, \emph{Forms in odd degree extensions
  and self-dual normal bases}, Amer. J. Math. {\bf112} (1990),
359--373.
%
\bibitem[BMT04]{BMT04}
G.~Berhuy, M.~Monsurr\`o, J.-P. Tignol, 
\emph{Cohomological invariants and R-triviality of adjoint classical
  groups}, Math. Z., {\bf 248} (2004), 313--323.
%
\bibitem[BFT07]{BFT07}
G. Berhuy, C. Frings, J.-P. Tignol, \emph{Galois cohomology of the
classical groups over imperfect fields}, J. Pure Appl. Algebra {\bf
  211} (2007), 307--341. 
%
\bibitem[BST93]{BST93}
E. Bayer-Fluckiger, D. B. Shapiro, J.-P. Tignol, \emph{Hyperbolic
  involutions}, Math. Z. {\bf 214} (1993), 461--476.   
%
\bibitem[CTS77]{CTS77}
J.-L.~Colliot-Th\'el\`ene, J.-J.~Sansuc, \emph{La $R$-\'equivalence
  sur les tores}, Ann. Sci. \'Ecole Normale Sup., 4\`eme s\'erie, {\bf
  10} (1977),  175--229.
%
\bibitem[GS17]{GS}
P.~Gille, T.~Szamuely,
\emph{Central simple algebras and Galois cohomology},
(2d edition), Cambridge studies in advanced math.~165, Cambridge
Univ. Press, Cambridge (2017).
%
\bibitem[KMRT98]{KMRT98}
M.-A. Knus,  S.~A. Merkurjev,  M.~Rost,  J.-P. Tignol,
\emph{The book of involutions},
Colloquium Publ., vol. {\bf 44}, AMS, Providence, RI,
(1998).
%
\bibitem[Lam05]{Lam05}
T. Y. Lam, {\it Introduction to quadratic forms over fields}, Graduate
Studies in Math. \textbf{67}, AMS, Providence,
RI, 2005.
%
\bibitem[Lew00]{Lew00}
D.~W. Lewis, \emph{The Witt semigroup of central simple algebras with
  involution}, Semigroup Forum, \textbf{60} (2000), 80--92.
%
\bibitem[Mer96]{Mer96}
A.~S. Merkurjev, 
\emph{R-equivalence and rationality problem for semisimple adjoint classical algebraic groups},
Publ. Math. IHES,
{\bf 84}
(1996),
189--213.
%
\bibitem[MT95]{MT95}
A.~S. Merkurjev, J.-P. Tignol,
\emph{The multipliers of similitudes and the Brauer group of homogeneous varieties},
J. reine angew. Math.,
{\bf 461}
(1995),
13--47.
%
\bibitem[PSS01]{PSS01}
R. Parimala, R. Sridharan, V. Suresh, \emph{Hermitian analogue of a
  theorem of Springer}, J. Algebra \textbf{243} (2001) 780--789.
%
\bibitem[QMT1]{QMT1}
A.~Qu\'eguiner-Mathieu, J.-P. Tignol,
\emph{Outer automorphisms of classical algebraic groups},
Ann.\ Sci.\ \'Ecole Normale
Sup.\ 4\`eme s\'erie, \textbf{51} (2018) 113--141.
%
\bibitem[QMT2]{QMT2}
A.~Qu\'eguiner-Mathieu, J.-P. Tignol,
\emph{Orthogonal involutions on central simple algebras and function
  fields of Severi--Brauer varieties}, Preprint.
%
\bibitem[Ros61]{Ros61}
M. Rosenlicht, Toroidal algebraic groups, Proc. Amer. Math. Soc. {\bf
  12} (1961), 984--988.
%
\bibitem[San81]{San81}
J.-J. Sansuc, Groupe de Brauer et arithm\'etique des groupes
alg\'ebriques lin\'eaires sur un corps de nombres, J. Reine
Angew. Math. {\bf 327} (1981), 12--80.
%
\bibitem[Ser68]{Ser68}
J-P.~Serre, \emph{Corps locaux} (2d edition), Hermann, Paris, 1968.
%
\bibitem[TW15]{TW15}
J.-P.~Tignol, A.R.~Wadsworth, \emph{Value Functions on Simple
  Algebras, and Associated Graded Rings}, Springer, Cham, 2015. 
%
\bibitem[VK85]{VK85}
V.~E. Voskresenski\u\i, A.~A. Klyachko, \emph{Toric Fano varieties and
  systems of roots}, Izv. Akad. Nauk SSSR Ser. Mat. {\bf 48} (1984)
237--263.  
\end{thebibliography}
\end{document}